\documentclass[a4paper,reqno,12pt]{amsart}

\usepackage{amsmath, amssymb, amsthm, enumerate,amsfonts,latexsym, mathrsfs}

\usepackage[top=30mm,right=27mm,bottom=30mm,left=27mm]{geometry}

\newtheorem{theorem}{Theorem}[section]

\newtheorem{lemma}[theorem]{Lemma}
\newtheorem{proposition}[theorem]{Proposition}

\newtheorem{defn}[theorem]{Definition}

\makeatletter
\newcommand{\imod}[1]{\allowbreak\mkern4mu({\operator@font mod}\,\,#1)}
\makeatother

\newtheorem*{CorA}{Corollary~A}

\newtheorem*{CorB}{Corollary~B}

\newtheorem*{CorC}{Corollary~C}
\newtheorem*{CorD}{Corollary~D}

\newtheorem{athm}{Theorem}

\theoremstyle{definition}
\newtheorem{example}[theorem]{Example}
\begin{document}

\title[\textbf{On some products of finite groups}]{\textbf{On some products of finite groups}}

\author[A. Ballester-Bolinches]{A. Ballester-Bolinches}
\address{Department of Mathematics, Guangdong University of Education, 510310, Guangzhou,
People’s Republic of China}
\address{Departament de Matem\`atiques, Universitat de Val\`encia, Dr. Moliner 50, 46100 Burjassot, Val\`encia, Spain}
\email{adolfo.ballester@uv.es}

\author[S. Y. Madanha]{S. Y. Madanha}
\address{Department of Mathematics and Applied Mathematics, University of Pretoria, Pretoria, 0002, South Africa}
\email{sesuai.madanha@up.ac.za}

\author[M. C. Pedraza-Aguilera]{M. C. Pedraza-Aguilera}
\address{Instituto Universitario de Matem\'atica Pura y Aplicada, Universitat Polit\`ecnica de Val\`encia, 46022 Camino de Vera,  Val\`encia, Spain}
\email{mpedraza@mat.upv.es}

\author[X. Wu]{X . Wu}
\address{School of Mathematics, Suzhou University, Suzhou 215006 Jiang, People’s Republic of China}
\email{wxwjs1991@126.com}

\thanks{}


\date{\today}


\begin{abstract}
A classical result of Baer states that a finite group $ G $ which is the product of two normal supersoluble subgroups is supersoluble if and only if $ G' $ is nilpotent. In this article we show that if $ G=AB $ is the product of supersoluble (respectively, $ w $-supersoluble) subgroups $ A $ and $ B $, $ A $ is normal in $ G $, $ B $ permutes with every maximal subgroup of each Sylow subgroup of $ A $, then  $ G $ is supersoluble (respectively, $ w $-supersoluble) provided that $ G' $ is nilpotent. We also investigate products of subgroups defined above when $ A\cap B=1 $ and obtain more general results.
\end{abstract}

\maketitle
\section{Introduction}

All groups considered here will be finite.

A significant number of articles investigating the properties of groups expressible as a product of two supersoluble subgroups were published since the $1957$ paper by Baer \cite{Bae57}  in which he proved that a normal product $G = AB$ of two supersoluble subgroups $A$ and $B$ is supersoluble provided that  the derived subgroup $G'$ is nilpotent. There has been many generalisations of this theorem. Instead of having normal subgroups, certain permutability conditions were imposed on the factors. The case in which $A$ permutes with every subgroup of $B$ and $B$ permutes with every subgroup of $A$, that is, when $G$ is a \emph{mutually permutable product} of $A$ and $B$ is in fact one of the most interesting cases and has been investigated in detail (see \cite{BBERA10}, for a thorough review of results in this context and also \cite{AFG92} for general results on products).

In this article we study a weak form of a normal product arising quite frequently in the structural study of mutually permutable products and appears not to have been investigated in detail.


\begin{defn}
Let $G=AB$ be a product of subgroups $A$ and $B$. We say that $G$ is a \emph{weak normal product} of $A$ and $B$ if
\begin{itemize}
\item[(a)] $A$ is normal in $G$.
\item[(b)] $B$ permutes with all the maximal subgroups of Sylow subgroups of $A$.
\end{itemize}
\end{defn}

As an important first step in the study of weak normal products $G = AB$, and motivated by the mutually permutable case, we analyse the situation $A \cap B=1$. In this case they are semidirect products of $A$ and $B$.

\begin{defn} Let the group $G = AB$ be the weak normal product of $A$ and $B$ with $A$ normal in $G$. We say that $ G $ is a \emph{weak direct product} of $A$ and $B$ if $A \cap B = 1$. In this case we write $G = [A]B$.

\end{defn}

We study these products when the factors are supersoluble and widely supersoluble and analyse the behaviour of the residuals associated to these classes of groups. Recall that a widely supersoluble group, or $w$-\emph{supersoluble} group for short, is defined as a group $ G $ such that every Sylow subgroup of $G$ is $\mathbb{P}$-subnormal in $G$ (a subgroup $H$ of a group $G$ is $\mathbb{P}$-subnormal in $G$ whenever either $H=G$ or there exists a chain of subgroups $H=H_{0} \leqslant H_{1} \leqslant \cdots \leqslant H_{n-1} \leqslant H_{n}=G$, such that $| H_{i} {:} H_{i-1}|$ is a prime for every $i=1, \dots, n$).

The class of $w$-supersoluble groups, denoted $w\mathfrak{U} $, is a subgroup-closed saturated formation containing the subgroup-closed saturated formation $\mathfrak{U} $ of all supersoluble groups. Moreover $w$-supersoluble groups have a Sylow tower of supersoluble type (see \cite[Corollary]{VVT10}). \\



Our first aim is to show that the saturated formations of all supersoluble groups and w-supersoluble groups are closed under the formation of weak direct products.
\begin{athm}\label{A}
 Let $ G =[A]B $ be a weak direct product of $ A $ and $ B $. If $ A $ and $ B $ belong to $ \mathfrak{U} $, then $ G $ is also supersoluble.
\end{athm}

\begin{CorA}\label{Gwsupersoluble}
Let $ G =[A]B $ be a weak direct product of $ A $ and $ B $. If $ A $ and $ B $ belong to $ w\mathfrak{U} $, then $ G $ is w-supersoluble.
\end{CorA}

Our second aim is to show that the product of the supersoluble (respectively, w-supersoluble) residuals of the factors of weak direct products is just the supersoluble (respectively, w-supersoluble) residual of the group.

\begin{athm}\label{B}
Suppose that $ \mathfrak{F} =\mathfrak{U}$ or $ \mathfrak{F} =w\mathfrak{U}$.  Let $ G =[A]B $ be a weak direct product of $ A $ and $ B $. Then
\begin{center}
$ G^{\mathfrak{F}}=A^{\mathfrak{F}}B^{\mathfrak{F}} $.
\end{center}
\end{athm}

We now analyse the behaviour of weak normal products with respect to the formations of all supersoluble and w-supersoluble groups. Our next result shows that Baer's theorem can be generalised in this new direction:
\begin{athm}\label{C}
 Let $ G=AB $ be a weak normal product of A and B. If $ G'$ is nilpotent, $A$ is normal in $G$, and $ A, B\in \mathfrak{U} $, then $ G \in \mathfrak{U}$.
\end{athm}
As a corollary, we obtain the result for $w \mathfrak{U} $-groups.
\begin{CorB}\label{G'nilpotentwsupersoluble}
 Let the group $G=AB$ be a weak normal product of $ w\mathfrak{U} $-subgroups $A$ and $B$. If $G'$ is nilpotent and $A$ is normal in $G$,  then $G$ belongs to $w \mathfrak{U} $.
\end{CorB}
Our second objective is to investigate the residuals of weak normal products. Unfortunately, it does not follow that $ G^{\mathfrak{U}}=A^{\mathfrak{U}}B^{\mathfrak{U}} $ when $ G $ is a weak normal product as the following example shows:

\begin{example}\label{counterweaknormal}
Let
\begin{align*}A=\langle g_2,g_4,g_5,g_6,g_7\mid \mbox{}&g_2^3=g_4^3=g_5^3=g_6^3=g_7^3=1,\\
                                                &g_4^{g_2}=g_4g_6, g_5^{g_2}=g_5g_7, g_6^{g_2}=g_6, g_7^{g_2}=g_7,\\
                                                &g_5^{g_4}=g_5, g_6^{g_4}=g_6, g_7^{g_4}=g_7, \\
  &g_6^{g_5}=g_6, g_7^{g_5}=g_7,\\& g_7^{g_6}=g_7\rangle.
\end{align*}
Let $Q=\langle b\rangle\cong C_4$ act on $A$ via
\begin{align*}
  g_2^b&=g_2,&g_4^{b}&=g_4g_5,&g_5^b&=g_4g_5^2,&g_6^{b}&=g_6g_7,&g_7^{b}&=g_6g_7^2.
\end{align*}
Let $G=[A]Q$ be the corresponding semidirect product.

Note that $A'=\Phi(A)=\langle g_6, g_7\rangle$. Let $A_0=\langle g_4, g_5\rangle$. Then $A_0$ is not a normal subgroup of $A$, but is normalised by $Q$. Let $B=A_0\langle b\rangle$, then $Core_G(B)=1$. Furthermore, $B$ permutes with the $13$ maximal subgroups of $A$. The supersoluble residual of $G$ is $\langle g_4, g_5, g_6,g_7\rangle$, giving a quotient isomorphic to $C_{12}$. Consequently $ G^{\mathfrak{U}}\neq A^{\mathfrak{U}}B^{\mathfrak{U}} $. This group corresponds to \texttt{SmallGroup(972, 406)} of GAP.
\end{example}

We prove the following:

\begin{athm}\label{D}
Let the group $G=AB$ be a  product of the subgroups $A$ and $B$. Assume that $A$ is a normal subgroup of $G$ and  every Sylow subgroup of $ B $ permutes with every maximal  subgroup of every Sylow subgroup  of $ A $. If $G'$ is nilpotent, then $ G^{\mathfrak{U}}=A^{\mathfrak{U}}B^{\mathfrak{U}} $.

\end{athm}

\begin{CorC}
 Let the group $G=AB$ be a  product of the subgroups $A$ and $B$. Assume that $A$ is a normal subgroup of $G$ and  every Sylow subgroup of $ B $ permutes with every maximal subgroup of every Sylow subgroup  of $ A $. If $G'$ is nilpotent, then $ G^{w \mathfrak{U}}=A^{w \mathfrak{U}}B^{w \mathfrak{U}} $.

\end{CorC}

Denote by $ \mathfrak{N} $ the class of all nilpotent groups. A nice result of Monakhov \cite[Theorem 1]{M18} states that if $ G=AB $ is the mutually permutable product of the supersoluble subgroups $ A $ and $ B $, then $ G^{\mathfrak{U}}=(G')^{\mathfrak{N}}=[A, B]^{\mathfrak{N}} $. We prove an analogue of this result for weak normal products.

\begin{CorD}\label{U-residual}
Let $G=AB$ be a weak normal product of the supersoluble subgroups $A$ and $B$. If $A$ is normal in $G$, we have that  $G^{\mathfrak{U}}=(G')^{\mathfrak{N}}=[A,B]^{\mathfrak{N}}$.
\end{CorD}

\section{Preliminary Results}
It is easy to see that factor groups of weak normal products are also weak normal products. For weak direct products we have the following:

\begin{lemma}\label{factor}
Let $ G =[A]B $ be a weak direct product of $ A $ and $ B $.
\begin{itemize}
\item[(a)] If  $ N $ is a normal subgroup of $ G $ such that $ N\leqslant A $ or $ N\leqslant B $, then $ G/N=[AN/N](BN/N) $ is a weak direct product of $AN/N$ and $BN/N$.
\item[(b)] If $ K $ is a subgroup of $ B $, then $ [A]K $ is a weak direct product of $ A $ and $ K $.
\end{itemize}
\end{lemma}
\begin{proof} (a) Let $H/N$ be a Sylow $p$-subgroup of $AN/N$. Then $H/N=PN/N$, where $P$ is a Sylow $p$-subgroup of $A$. Let $K/N$ be a maximal subgroup of $H/N$. Then $K=K\cap PN=N(P\cap K)$ and $K/N=N(P\cap K)/N$. Thus $$p=|PN/N:(P\cap K)N/N|=\frac{|P||N|}{|P\cap N|} \frac{|P\cap K\cap N|}{|P\cap K| |N|}=|P: P\cap K|.$$  Hence $P\cap K$ is a maximal subgroup of $P$. Then $B$ permutes with $P\cap K$ and so $BN/N$ permutes with $K/N$. Therefore $ G/N=[AN/N](BN/N) $ is a weak direct product of $AN/N$ and $BN/N$.

(b) Let $K$ be any proper subgroup of $B$ and $H$ be any maximal subgroup of a Sylow subgroup of $A$. By the hypotheses, we have $HB = BH$ and so $ H = H(A \cap B) = A \cap HB $. Since $ A  $ is normal in $G$, it follows that $H$ is normal in $HB $ and so $ B$  normalizes $H$. Hence $ K $ permutes with $ H $. Therefore $ [A]K $ is a weak direct product of $ A $ and $ K $.
\end{proof}

Our second lemma contains some of the properties of $\mathbb{P}$-subnormal subgroups.
\begin{lemma}\label{subnormal}\cite[Lemma 1.4]{VVT10}
Let $G$ be a soluble group and $H$ and $K$ two  subgroups of $G$. The following properties hold:
\begin{itemize}
\item[(i)] If $H$ is $\mathbb{P}$-subnormal in $G$ and $N$ is normal in $G$ then $HN/N$ is $\mathbb{P}$-subnormal $G/N$.
\item[(ii)] If $N$ is normal in $G$ and $HN/N$ is $\mathbb{P}$-subnormal in $G/N$ then $HN$ is $\mathbb{P}$-subnormal in $G$.
\item[(iii)] If $H$ is $\mathbb{P}$-subnormal in $K$ and $K$ is $\mathbb{P}$-subnormal in $G$ then $H$ is $\mathbb{P}$-subnormal in $G$.
\end{itemize}
\end{lemma}

\section{Supersoluble and w-supersoluble residuals }
We start this section by proving Theorem A.

\begin{proof}[\textbf{Proof of Theorem A}]
Assume that the result is false and let $G$ be a counterexample of minimal order. Clearly, $G$ is soluble and $A$ and $B$ are proper subgroups of $G$. Let $N$ be a minimal normal subgroup of $G$ contained in $A$, then applying Lemma~\ref{factor}(a), $G/N=[A/N] (BN/N)$ is a weak direct product of $A/N$ and $BN/N$.
By the minimality of $G$, $G/N \in \mathfrak{U}$. Since the class of all supersoluble groups is a saturated formation, there exists a unique minimal normal subgroup $N$ of $G$ contained in $A$, $N$ a $p$-group for some prime $p$, $|N| > p$, and $\Phi(A) = 1$. Since $A$ is supersoluble, $A$ has a normal Sylow subgroup and since $N$ is the unique minimal normal subgroup of $G$ contained in $A$, it follows that $\textbf{F}(A)$ is a $p$-group and $\textbf{F}(A)$ is an elementary abelian Sylow $p$-subgroup of $A$.

Assume that $A$ is not a $p$-group. Then $\textbf{F}(A)$ is a completely reducible $A$-module and so $ \textbf{F}(A) = N \times Z$, for some $A$-module $Z$. Let $L$ be a minimal normal subgroup of $A$ contained in $N$. Then $N = L \times D$, for some $A$-module $D$. Then $\textbf{F}(A) = L \times DZ$, and $ DZ $ is a maximal subgroup of $\textbf{F}(A)$ because $L$ is of prime order. Therefore $ E = DZ$ permutes with $B$. Hence $ DZ=A\cap (DZ)B $ and so $ DZ $ is normalised by $B$. Since $DZ$ is also normalised by $A$, it follows that $DZ$ is a normal subgroup of $G$. The minimality of $N$ forces $D = 1$ and so $N$ is of prime order, which is a contradiction. Consequently, $A$ is an elementary abelian $p$-group. Note that $A$ cannot be cyclic since $|N| > p$. Let $1 \neq X$ be a maximal subgroup of $A$. Arguing as above, we have that $X$ is normal in $XB$ so that $X$ is normalised by $B$. Hence $X$ is normal in $G$ because $A$ is abelian. Therefore $N$ is contained in $X$ and so $N \leq \Phi(A)= 1$, our final contradiction.
\end{proof}

\begin{proof}[\textbf{Proof of Corollary A}]
Assume, by way of contradiction, that the result fails, and let $G$ be a counterexample of least order. Clearly $G$ is soluble and $A$ and $B$ are proper subgroups of $G$. Since the class of all w-supersoluble groups is a saturated formation, we can argue as in Theorem~\ref{A} to conclude that there exists a unique  minimal normal subgroup $N$ of $G$ contained in $A$, $N$, and $N$ is a $p$-group for some prime $p$. Moreover, $\Phi(A) = 1$, and $A_p = \textbf{F}(A)$ is the Sylow $p$-subgroup of $A$.  By the minimality of $G$, $G/N$ is $w$-supersoluble. Let $P$ be a Sylow $p$-subgroup of $G$. Then $P/N$ is $\mathbb{P}$-subnormal in $G/N$. By Lemma \ref{subnormal}(ii), $P$ is $\mathbb{P}$-subnormal in $G$. Suppose that for every prime $q \neq p$ dividing $|G|$ and every Sylow $q$-subgroup $B_{q}$ of $B$ we have that $AB_{q}$ is a proper subgroup of $G$. Let $A_{q}$ be a Sylow $q$-subgroup of $A$ such that $G_{q}=A_{q}B_{q}$ is a Sylow $q$-subgroup of $G$. Since $G/N$ is $w$-supersoluble, it follows that $G_{q}N$ is $\mathbb{P}$-subnormal in $G$. By Lemma \ref{factor}(b), $AB_{q}$ satisfies the hypotheses of the theorem. Hence $AB_{q}$ is $w$-supersoluble by the choice of $G$. Thus $G_{q}N \leq AB_{q}$ is $w$-supersoluble. Consequently, $G_{q}$ is $\mathbb{P}$-subnormal in $G_{q}N$ which is $\mathbb{P}$-subnormal in $G$. Applying Lemma \ref{subnormal}(iii), $G_{q}$ is $\mathbb{P}$-subnormal in $G$. Therefore the Sylow subgroups of $G$ are $\mathbb{P}$-subnormal in $G$ and so $G$ is $w$-supersoluble, a contradiction. Thus we may assume there exists $q \neq p$ such that $G=AB_{q}$.  Let $T=A_{p}G_{q}=(A_{p}A_{q})B_{q}$. Since $A$ is normal in $G$, we have that $A_{q}$ is normal in $G_{q}$ and then $A_{p}A_{q}$ is normalized by $B_{q}$. Moreover, $A_{p}A_{q}$ is a $w$-supersoluble metanilpotent subgroup of $G$. By \cite[Theorem 2.13(1)]{VVT10}, $A_{p}A_{q}$ is supersoluble. It is clear that $T$ is a weak direct product of the supersoluble subgroups $A_{p}A_{q}$ and $B_{q}$. Applying Theorem~\ref{A}, it follows that $T$ is supersoluble. Therefore $T$ is $w$-supersoluble. But $G_{q}N \leq T$ which is $w$-supersoluble. Thus $G_{q}$ is $\mathbb{P}$-subnormal in $G_{q}N$ which is $\mathbb{P}$-subnormal in $G$. Again the application of Lemma \ref{subnormal}(iii) yields  $G_{q}$ is $\mathbb{P}$-subnormal in $G$. If $G_{r}$ is a Sylow $r$-subgroup of $G$ for some prime $r \neq p, q$, then $G_{r}$ is contained in $A$ and so $G_{r}$ is $\mathbb{P}$-subnormal in $A$. Since $A$ is also $\mathbb{P}$-subnormal in $G$, we have that $G_{r}$ is $\mathbb{P}$-subnormal in $G$. Consequently, every Sylow subgroup of $G$ is $\mathbb{P}$-subnormal in $G$ and $G$ is w-supersoluble. This final contradiction completes the proof of the corollary.

\end{proof}

\begin{proof}[\textbf{Proof of Theorem B}]
Suppose that the result is not true and let $ G $ be a minimal counterexample. Then\\ \\
\textit{(i) $ A \in \mathfrak{F} $, $ B^{\mathfrak{F}}\not= 1 $, $ Core_{G}(B)=1 $ and $ G^{\mathfrak{F}}=B^{\mathfrak{F}}N $ for every minimal normal subgroup $ N $ of $ G $ such that $ N\leqslant A $.}

 Let $ N $ be a minimal normal subgroup of $ G $ such that $ N\leqslant A $ or $ N\leqslant B $. Then $ G/N=[AN/N](BN/N) $ is a weak direct product of $AN/N$ and $BN/N$ by Lemma~\ref{factor}(a). The minimal choice of $G$ implies that $ G^{\mathfrak{F}}N/N=(A^{\mathfrak{F}}N/N) (B^{\mathfrak{F}}N/N) $, that is, $ G^{\mathfrak{F}}N=A^{\mathfrak{F}}B^{\mathfrak{F}}N $. Since $ G/G^{\mathfrak{F}}\in \mathfrak{F} $, $ AG^{\mathfrak{F}}/G^{\mathfrak{F}} $ and $ BG^{\mathfrak{F}}/G^{\mathfrak{F}} $ also belong to $ \mathfrak{F} $ and then $ A^{\mathfrak{F}}\leqslant G^{\mathfrak{F}} $ and $ B^{\mathfrak{F}}\leqslant G^{\mathfrak{F}} $. If $ G^{\mathfrak{F}}\cap N= 1 $, then $ G^{\mathfrak{F}}=A^{\mathfrak{F}}B^{\mathfrak{F}}(G^{\mathfrak{F}}\cap N)=A^{\mathfrak{F}}B^{\mathfrak{F}} $, a contradiction. Hence $ G^{\mathfrak{F}}=A^{\mathfrak{F}}B^{\mathfrak{F}} N $ for every minimal normal subgroup $ N $ of $ G $ such that $ N\leqslant A $ or $ N\leqslant B $. If $A^{\mathfrak{F}} \neq 1$, then there exists a minimal normal subgroup $N$ of $G$ contained in $A^{\mathfrak{F}}$ because $A^{\mathfrak{F}}$ is normal in $G$. This contradiction yields $A \in \mathfrak{F}$ and $ G^{\mathfrak{F}}=B^{\mathfrak{F}}N $ for every minimal normal subgroup $ N $ of $ G $ such that $ N\leqslant A $ or $ N\leqslant B $.  If $ B \in \mathfrak{F} $, then $ G\in \mathfrak{F} $ by Theorem~\ref{A} and Corollary~A, contrary to assumption.  Hence $ B^{\mathfrak{F}} \neq 1 $.
 Suppose that $Core_{G}(B) \neq 1$. Let $N$ be a minimal normal subgroup of $G$ contained in $B$ and let $R$ be a minimal normal subgroup of $G$ contained in $A$. Then $G^{\mathfrak{F}}=B^{\mathfrak{F}}N \cap B^{\mathfrak{F}}R \leq B \cap B^{\mathfrak{F}}R=B^{\mathfrak{F}}(B \cap R)=B^{\mathfrak{F}}$, a contradiction. Consequently we have that $Core_{G}(B)=1$. \\ \\

\textit{(ii) $ \mathbf{F}(A)$ is a Sylow $p$-subgroup of $A$, where $ p $ is the largest prime dividing $ |A| $ .}

Since  $ A \in \mathfrak{F} $, it follows that $A$ is a Sylow tower group of supersoluble type. In particular, $ 1 \neq \mathbf{O}_{p}(A)$ is the Sylow $p$-subgroup of $A$, where $ p $ is the largest prime dividing $ |A| $.  If $ \mathbf{F}(A) $ is not a $ p $-group, then $1 \neq  \mathbf{O}_{q}(A)\leqslant \mathbf{O}_{q}(G)$. Let $ N_{1} $ be a minimal normal subgroup of $ G $ contained in $\mathbf{O}_{p}(A) $ and let $ N_{2} $ be a minimal normal subgroup of $ G $ contained in $\mathbf{O}_{q}(A) $. Then $ G^{\mathfrak{F}}=B^{\mathfrak{F}}N_{1}=B^{\mathfrak{F}}N_{2}$, which is a contradiction since $ B^{\mathfrak{F}}\cap N_{1}=B^{\mathfrak{F}}\cap N_{2}=1 $. Therefore $ \mathbf{F}(A) = \mathbf{O}_{p}(A)$ is the Sylow $p$-subgroup of $A$.\\ \\

\textit{(iii) $ G $ is soluble, $ AK $ belongs to $ \mathfrak{F} $ for every proper subgroup $ K $ of $ B $; in particular, $B$ is a minimal non-supersoluble group and $ B^{\mathfrak{F}} $ is a $ q $-subgroup of $ B $ for some prime $ q $.}

Suppose that $ K $ is a proper subgroup of $ B $. By Lemma \ref{factor}, $AK$ satisfies the hypotheses of the theorem and so $ (AK)^{\mathfrak{F}}=K^{\mathfrak{F}} $ by the minimal choice of $G$. Since $ (AK^{x})^{\mathfrak{F}}=(K^{x})^{\mathfrak {F}}=(K^{\mathfrak{F}})^{x} $ for any $ x\in B $, it follows that $ A $ normalizes $ (K^{\mathfrak{F}})^{x} $. Thus $ A $ normalizes $ \langle (K^{\mathfrak{F}})^{x} \mid x\in B \rangle $. Then $ \langle (K^{\mathfrak{F}})^{x} \mid x\in B \rangle \lhd G $, contrary to $ Core_{G}(B)=1 $. Hence $ (K^{\mathfrak{F}})^{x}=1 $. Consequently, $ AK $ belongs to $ \mathfrak{F} $. This shows that $ B $ is $ \mathfrak{F} $-critical and by \cite[Theorem 2.9]{VVT10}, we have that $ B $ is a minimal non-supersoluble group. By \cite[Theorem~10]{BBER07}, we have that $ B^{\mathfrak{F}} $ is a  $ q $-group for some prime $ q $. In particular, $B$ and then $G$ are soluble.\\ \\

\textit{(iv) $ G^{\mathfrak{F}}=B^{\mathfrak{F}}\times N $ is an elementary abelian $ p $-group.}

Applying (iii), it follows that $B^{\mathfrak{F}}$ is a $q$-group for some prime $q$. Let $N$ be a minimal normal subgroup of $G$ contained in $A$. Then $G^{\mathfrak{F}}=B^{\mathfrak{F}}N$ by (i), and $N$ is a $p$-group by (ii).

Suppose that $B^{\mathfrak{F}}$ is a normal subgroup of $G^{\mathfrak{F}}$. Then $G^{\mathfrak{F}}/B^{\mathfrak{F}}$ is an elementary abelian $p$-group. Consequently, the residual $X$ of $G^{\mathfrak{F}}$ associated to the formation of all elementary abelian $p$-groups is a normal subgroup of $G$ contained in $B$. Hence $X \leq Core_G(B) = 1$ and $G^{\mathfrak{F}}$ is an elementary abelian $p$-group.

Assume that $p \neq q$. Let $N$ be a minimal normal subgroup of $G$ contained in $A$. Then $G^{\mathfrak{F}}=B^{\mathfrak{F}}N$ and $N$ is a $p$-group by (ii). Hence $B^{\mathfrak{F}}$ is a Sylow $q$-subgroup of  $ G^{\mathfrak{F}}=B^{\mathfrak{F}}N $. Applying Frattini's argument, we have that $G=G^{\mathfrak{F}}N_{G}(B^{\mathfrak{F}})=NN_{G}(B^{\mathfrak{F}})$. Since $Core_{G}(B)=1$, it follows that $N_{G}(B^{\mathfrak{F}})$ is a proper subgroup of $G$. Hence $N$ is not contained in $\Phi(G)$ for each minimal normal subgroup $N$ of $G$ contained in $A$. If $\Phi(A) \neq 1$, a minimal normal subgroup of $G$ must be contained in $\Phi(A) \leq \Phi(G)$, a contradiction.  Therefore $\Phi(A)=1$. Let $N$ be a minimal normal subgroup of $G$ contained in $A$. Then $N=N_{1} \times N_{2} \times \cdots \times N_{r}$ is a direct product of minimal normal subgroups of $A$, and there exists $i \in \{1,2, \dots, r \}$ such that $N_{i}$ is not contained in $ \Phi(A)$. Suppose $i=1$. Let $M$ be a maximal subgroup of $A$ such that $A=N_{1}M$ and $N_{1} \cap M=1$. Assume first that $A$ is a $p$-group. Then $BM$ is a subgroup of $G$, and $M=BM \cap A$ is a normal subgroup of $BM$. Hence $M$ is normalized by $B$ and so $M$ is a normal subgroup of $G$. Now $N=N_{1}(M \cap N)$. But $M \cap N$ is normal in $G$. The minimality of $N$ yields $N=N_{1}$ and then $|N|=p$. Thus $G/C_{G}(N)$ is abelian. Hence $G^{\mathfrak{F}}$ centralises $N$, and $B^{\mathfrak{F}}$ is a normal subgroup in $G^{\mathfrak{F}}$ and so $G^{\mathfrak{F}}$ is an elementary abelian $p$-group. This contradiction implies that $A$ is not a $p$-group. Then $T= \mathbf{F}(A)B$ is a proper subgroup of $G$ which is a weak direct product of $ \mathbf{F}(A)$ and $B$. By the minimality of $G$, $T^{\mathfrak{F}}=B^{\mathfrak{F}}$. Then  $B^{\mathfrak{F}}$ is a normal subgroup of $G^{\mathfrak{F}}$ and so $G^{\mathfrak{F}}$ is an elementary abelian $p$-group, a contradiction which shows that $p = q$. Then $B^{\mathfrak{F}}$ is a subnormal subgroup of $G$. By \cite[Lemma~A.14.3]{DH92}, $N$ normalises $B^{\mathfrak{F}}$ and therefore  $B^{\mathfrak{F}}$ is a normal subgroup of the elementary abelian $p$-group $G^{\mathfrak{F}}$.\\ \\

\textit{(v) Final contradiction.}
By \cite[IV, 5.18]{DH92}, since $ B^{\mathfrak{F}} $ is abelian, there exists an $ \mathfrak{F} $-projector $ K $ of $ B $ such that $ B=B^{\mathfrak{F}}K $ and $ K\cap B^{\mathfrak{F}}=1 $. Consider the subgroup $Z=AK$ of $G$. Applying (iii), $Z$ belongs to $ \mathfrak{F} $ and $G=B^{\mathfrak{F}}Z=F(G)Z$. By \cite[III, 3.23(b)]{DH92}, there exists a unique $\mathfrak{F}$-projector of $G$ containing $Z$, $E$ say. Hence $G=B^{\mathfrak{F}}Z=G^{\mathfrak{F}}E$ and $G^{\mathfrak{F}} \cap E=1$ by (iii) and \cite[IV, 5.18]{DH92}. In particular,  $B^{\mathfrak{F}} \cap Z=1$. Now $|Z||B^{\mathfrak{F}}|=|E| |G^{\mathfrak{F}} |=|E||B^{\mathfrak{F}}| |N |$. Hence $|Z |=|E| |N| $. This implies $Z=E$ and then $B^{\mathfrak{F}}=G^{\mathfrak{F}}$, a contradiction.
\end{proof}

\begin{proof}[\textbf{Proof of Theorem C}]
Assume that the result is false and let $G$ be a minimal counterexample. Then every proper epimorphic image of $G$ is supersoluble, and hence $G$ has exactly one minimal normal subgroup $N$ which is not contained in the Frattini subgroup of $G$. Since $G$ is soluble, it follows that $N$ is abelian, $N=C_{G}(N)=F(G)$ and there exists a core-free maximal subgroup of $G$ such that $G=NM$ and $N \cap M = 1$. Let $p$ be the prime dividing $|N|$. Then $|N| > p$. Since $1 \neq G'$ is nilpotent, we have that $G'=N$ and $M$ is abelian. But $\mathbf{O}_{p}(M) = 1$ by \cite[Lemma~A.13.6]{DH92}. Hence $M$ is a $p'$-group and $N$ is the Sylow $p$-subgroup of $G$. Since $B \neq G$, we have that $N \leq A$. Note that $N = C_{A}(N) = \mathbf{O}_{p'p}(A)$. Therefore $A/\mathbf{O}_{p'p}(A)=A/\mathbf{O}_{p}(A)= A/N$ is abelian of exponent dividing $p-1$ because $A$ is supersoluble. Assume $BN$ is a proper subgroup of $G$. Then, by the minimality of $G$, $BN$ is supersoluble, and so $B_{p'} \cong BN/\mathbf{O}_{p'p}(BN)$ is abelian of exponent dividing $p-1$. Consequently $M$ is abelian of exponent dividing $p-1$. Since $N$ is an irreducible and faithful module for $M$, we have that $N$ has order $p$ by \cite[Theorem~B.9.8]{DH92}, a contradiction. Hence $G=BN$. Now $B \cap N$ is a normal subgroup of $G$ contained in $N$. Thus $B \cap N=1$, and $G=BN$ is the weak direct product of $B$ and $N$. By Theorem \ref{A}, $G$ is supersoluble. This contradiction proves the theorem.
\end{proof}

\begin{proof}[\textbf{Proof of Corollary B}]
Note that since $ G' $ is nilpotent, $ A $ and $ B $ are metanilpotent. By \cite[Theorem 2.11]{VVT10}, $ A $ and $ B $ are supersoluble. By Theorem~\ref{C}, $ G $ is supersoluble and hence $ G\in w \mathfrak{U} $.
\end{proof}

\begin{proof}[\textbf{Proof of Theorem D}]
Suppose the theorem is not true and let $ (G, A, B) $ be a counterexample with $ |G| + |A| + |B| $ as small as possible. Let $ N $ be a minimal normal subgroup of $ G $. It is easy to check that $ G/N $ satisfies the hypotheses of the theorem. By the minimality of $ G $, we have that $ G^{\mathfrak{U}}N=A^{\mathfrak{U}}B^{\mathfrak{U}}N $. Hence $ G^{\mathfrak{U}}=A^{\mathfrak{U}}B^{\mathfrak{U}}(G^{\mathfrak{U}}\cap N) $. Consequently, $ Soc(G) $ is contained in $ G^{\mathfrak{U}} $ and $ G^{\mathfrak{U}}=A^{\mathfrak{U}}B^{\mathfrak{U}}N $ for every minimal normal subgroup $ N $ of $ G $. Since $ G^{\mathfrak{U}}$ is contained in $G'$, we have that $ G^{\mathfrak{U}}$ is nilpotent.

Note that $A^{\mathfrak{U}}$ is a normal subgroup of $G$. If $ A^{\mathfrak{U}}\not= 1 $, then there exists a minimal normal subgroup $N$ of $G$ such that $N \leq  A^{\mathfrak{U}} $  and so $ G^{\mathfrak{U}}=A^{\mathfrak{U}}B^{\mathfrak{U}}N=A^{\mathfrak{U}}B^{\mathfrak{U}} $, a contradiction. Hence we may assume that $ A $ is supersoluble and that $ G^{\mathfrak{U}}=B^{\mathfrak{U}}N $ for every minimal normal subgroup $N$ of $G$. If $B$ were supersoluble, then $G$ would be supersoluble by Theorem~\ref{C}, which is a contradiction. Hence $ B^{\mathfrak{U}} \neq 1$. Furthermore, $ B^{\mathfrak{U}} $ cannot contain a normal subgroup of $ G $. Hence $ Core_{G}(B^{\mathfrak{U}})=1 $. Let $p$ be the largest prime dividing $ |A| $. Since $A$ is a Sylow tower group of supersoluble type, $A$ has a normal Sylow $p$-subgroup, $ A_{p} $ say, which is also normal in $G$. Hence $G$ has a minimal normal subgroup $ N $ of $ G $ which is a $ p $-group.  Since $G^{\mathfrak{U}}$ is nilpotent, we have that $B^{\mathfrak{U}}$ is a subnormal subgroup of $G$. By \cite[Lemma~A.14.3]{DH92}, $B^{\mathfrak{U}}$ is normalized by $N$. Thus $B^{\mathfrak{U}}$ is a normal subgroup of $G^{\mathfrak{U}}$ and $ G^{\mathfrak{U}}/B^{\mathfrak{U}} $ is an elementary abelian $ p $-group. Consequently $ B^{\mathfrak{U}} $ contains the residual $X$ of $ G^{\mathfrak{U}}$ associated to the formation of all elementary abelian $p$-groups. Since $X$ is a normal subgroup of $G$, it follows that $X \leq  Core_{G}(B^{\mathfrak{U}})=1$. Hence $ G^{\mathfrak{U}} $ is an elementary abelian $ p $-group.

Since $ Soc(G) $ is contained in $ G^{\mathfrak{U}} $, $ \textbf{O}_{p'}(G)=1 $ and hence $ \mathbf{F}(G)=\textbf{O}_{p}(G) $. Therefore $ G' \leq \mathbf{F}(G) $ is a $ p $-group and $ \mathbf{F}(G) $ is the unique  Sylow $ p $-subgroup of $ G $. Moreover the Hall $p'$-subgroups of $G$ are abelian (note that $G$ is soluble). Assume $ A_{p}B < G $. Then $ A_{p}B $ satisfies the hypotheses of the theorem. By the choice of $ G $, we have that $ (A_{p}B)^{\mathfrak{U}}=B^{\mathfrak{U}} $. Note that $G' \leq A_{p}B $. Hence $A_{p}B$ is a normal subgroup of $ G $. This implies that $ B^{\mathfrak{U}} $ is normal in $ G $, a contradiction. Hence $ G=A_{p}B$ and $A_{p}$ and $B$ satisfy the hypotheses of the theorem. If $A \neq A_{p}$, the choice of  $(G, A, B)$ implies that $ G^{\mathfrak{U}}=B^{\mathfrak{U}} $, a contradiction. Consequently we have that  $A=A_{p}$.


Write $T=AB_{p'}$.  By Theorem~\ref{C}, $T$ is supersoluble. Moreover, since $G=F(G)B_{p'}$, it follows that every minimal normal subgroup $N$ of $G$ contained in $T$ is a minimal normal subgroup of $T$. Thus $|N|=p$. Consequently, $N$ is ${\mathfrak{U}}$-central in $G$. By \cite[V, 3.2]{DH92}, $N$ is contained in every supersoluble normaliser of $G$. Let $E$ be one of them. Then $G=G^{\mathfrak{U}}E$ and $G^{\mathfrak{U}} \cap E=1$. However, $N \leq G^{\mathfrak{U}} \cap E=1$. This final contradiction proves the theorem.



\end{proof}

\begin{proof}[\textbf{Proof of Corollary C}]
Since $\mathfrak{U}\subseteq  w\mathfrak{U}$, we have $G^{w\mathfrak{U}} \leq G^{\mathfrak{U}} \leq G'$. Then $G/G^{w\mathfrak{U}}$ is a metanilpotent w-supersoluble group. Applying \cite[Theorem 2.11]{VVT10}, we have that $G/G^{w\mathfrak{U}}$ is supersoluble. Hence $G^{\mathfrak{U}} \leq G^{w \mathfrak{U}}$, and therefore $G^{\mathfrak{U}} = G^{w \mathfrak{U}}$ and the same is true for $A$ and $B$. Therefore, by Theorem~\ref{D}, $ G^{w \mathfrak{U}}=A^{w \mathfrak{U}}B^{w \mathfrak{U}} $, as desired.
\end{proof}

\section{An analogue of Monakhov's result}

The following two results are the key to prove Corollary~D.

\begin{lemma} \cite[Theorem A]{BBPA18}\\
Let the group $G=HK$ be the product of the subgroups $H$ and $K$. Assume that $H$ permutes with every maximal subgroup of $K$ and $K$ permutes with every maximal subgroup of $H$. If $H$ is supersoluble, $K$ is nilpotent and $K$ is $\delta$-permutable in $H$, where $\delta$ is a complete set of Sylow subgroups of $H$, then $G$ is supersoluble.
\end{lemma}
\begin{proposition}\label{subnormalderived}
Let $G=AB$, be a weak normal product of $A$ and $B$ with $A$ and $B$ supersoluble and $A$ normal in $G$. Then $B'$ is a subnormal subgroup of $G$.
\end{proposition}

\begin{proof}
Assume the result is not true and let $G$ be a counterexample of minimal order with $|A|$ as small as possible. Let $p$ be the largest prime dividing the order of $A$. Then $A$ has a normal Sylow $p$-subgroup $A_{p}$ which is also a normal subgroup of $G$. Let $N$ be a minimal normal subgroup of $G$ such that $N \leq A_{p}$. It is clear that $A_{p}B$ satisfies the hypotheses of the theorem. Assume that $A_{p}B$ is a proper subgroup of $G$. By the minimality of $G$, $B'$ is a subnormal subgroup of $A_{p}B$. Hence $B' \leq F(A_{p}B)$. By Lemma~\ref{factor}(a), $G/N$ is a weak normal product of $A/N$ and $BN/N$. By the minimality of $G$ we have that $B'N$ is a subnormal subgroup of $G$. Since $N \leq F(A_{p}B)$, it follows that $B'N \leq F(A_{p'}B)$. Hence $B'N$ is a subnormal nilpotent subgroup of $G$. Consequently $B'N \leq F(G)$. Thus $B'$ is a subnormal subgroup of $G$, a contradiction. Hence we may assume that $G=A_{p}B$. The minimality of $|A|$ implies $A=A_{p}$. Applying now the above Lemma, we conclude that  $G$ is supersoluble and therefore $G'$ is nilpotent. Hence $B'$ is subnormal in $G$. This final contradiction proves the proposition.
\end{proof}

\begin{proof}[\textbf{Proof of Corollary D}]
Arguing as in \cite[Theorem 1]{M18}, we obtain $G^{\mathfrak{U}}=(G')^{\mathfrak{N}}$. Moreover, by \cite[Lemma 1(3)]{M18} we have that $G'=A'B'[A,B]=(A')^{G}(B')^{G}[A,B]$. Since $A$ is a normal subgroup of $G$, then $A'$ is a subnormal subgroup of $G$. Also the application of Proposition \ref{subnormalderived} yields $B'$ is subnormal in $G$ and both $A'$ and $B'$ are nilpotent. Hence $(A')^{G}(B')^{G}$ is a normal nilpotent subgroup of $G$. By \cite[II, Lemma~II.2.12]{DH92}, $(G')^{\mathfrak{N}}=((A')^{G}(B')^{G})^{\mathfrak{N}}[A,B]^{\mathfrak{N}}=[A,B]^{\mathfrak{N}}$, as desired.
\end{proof}

\section{Acknowledgements}
These results are part of the R+D+i project supported by the Grant
PGC2018-095140-B-I00, funded by MCIN/AEI/10.13039/501100011033 and by ``ERDF A way of making Europe'' and
the Grant PROMETEO/2017/057
funded by GVA/10.13039/501100003359.

\end{document}